\documentclass[11pt]{amsart}
\usepackage{tgpagella}
\usepackage{euler}
\usepackage[T1]{fontenc}
\usepackage{amsmath, amssymb}
\usepackage[hidelinks]{hyperref}
\usepackage[english]{babel}
\usepackage{mathrsfs}
\usepackage{tikz-cd}
\usepackage{eucal}
\usepackage[all]{xy}
\newtheorem{thm}{Theorem}[section]
\newtheorem*{theorem}{Theorem}

\newtheorem{defn}[thm]{Definition}

\newtheorem{lemma}[thm]{Lemma}
\newtheorem{prop}[thm]{Proposition}
\newtheorem{cor}[thm]{Corollary}
\newtheorem*{corollary}{Corollary}

\newtheorem{fact}[thm]{Fact}
\newtheorem{question}[thm]{Question}
\newcommand{\dminus}{ 
\buildrel\textstyle\ .\over{\hbox{ 
\vrule height3pt depth0pt width0pt}{\smash-} 
}}

\def \u{\mathcal U}

\newcommand{\cstar}{$\mathrm{C}^*$}
\def \val{\operatorname{val}}
\def \ab{\operatorname{ab}}

\makeatletter

\def\indsym#1#2{%
  \setbox0=\hbox{$\m@th#1x$}%
  \kern\wd0%
  \hbox to 0pt{\hss$\m@th#1\mid$\hbox to 0pt{$\m@th#1^{#2}$}\hss}%
  \lower.9\ht0\hbox to 0pt{\hss$\m@th#1\smile$\hss}%
  \kern\wd0}

\def\nindsym#1#2{%
  \setbox0=\hbox{$\m@th#1x$}%
  \kern\wd0%
  \hbox to 0pt{\hss$\m@th#1\not$\kern1.4\wd0\hss}
  \hbox to 0pt{\hss$\m@th#1\mid$\hbox to 0pt{$\m@th#1^{\,#2}$}\hss}%
  \lower.9\ht0\hbox to 0pt{\hss$\m@th#1\smile$\hss}%
  \kern\wd0}

\def\dotminussym#1#2{%
  \setbox0=\hbox{$\m@th#1-$}%
  \kern.5\wd0%
  \hbox to 0pt{\hss\hbox{$\m@th#1-$}\hss}%
  \raise.6\ht0\hbox to 0pt{\hss$\m@th#1.$\hss}%
  \kern.5\wd0}
\newcommand{\dotminus}{\mathbin{\mathpalette\dotminussym{}}}

\def \Th{\operatorname{Th}}

\def \val{\operatorname{val}}

\def \id{\operatorname{id}}

\newcommand{\cqs}{C_q}
\newcommand{\cqa}{C_{qa}}
\newcommand{\cqc}{C_{qc}}
\title{On definability of \cstar-tensor norms}
\author{Isaac Goldbring}
\thanks{I. Goldbring was partially supported by NSF grant DMS-2054477. }

\address{Department of Mathematics\\University of California, Irvine, 340 Rowland Hall (Bldg.\# 400),
Irvine, CA 92697-3875}
\email{isaac@math.uci.edu}
\urladdr{http://www.math.uci.edu/~isaac}

\author{Thomas Sinclair}
\thanks{T.~Sinclair was partially supported by NSF grant DMS-2055155.}

\address{Mathematics Department, Purdue University, 150 N. University Street, West Lafayette, IN 47907-2067}
\email{tsincla@purdue.edu}
\urladdr{http://www.math.purdue.edu/~tsincla/}

\begin{document}

\begin{abstract}
We initiate the study of definability (in the model-theoretic sense) of \cstar-tensor norms.  We show that neither the minimal nor maximal tensor norms are definable uniformly over all \cstar-algebras.  The proof in the case of the minimal tensor product norm uses a deep theorem of Kirchberg characterizing exactness in terms of tensor products with matrix ultraproducts while the case of maximal tensor products uses Pisier's recent characterization of the lifting property in terms of maximal tensor products and ultraproducts.  We also study the question of when one of these tensor products can be definable in a particular \cstar-algebra.  We establish some negative results along these lines for particular \cstar-algebras and when the definability condition is strengthened to be computable and of a restricted quantifier-complexity; these results use the quantum complexity results MIP$^*$=RE and MIP$^{co}=$coRE.  As a byproduct of our arguments, we answer a question of Fritz, Netzer, and Thom by showing that the norm on $C^*(\mathbb{F}_n\times \mathbb{F}_n)$ is not computable for any $n\in \{2,3,\ldots,\infty\}$.
\end{abstract}

\maketitle

\section{Introduction}

One of the more interesting aspects of \cstar-algebra theory is the problem of defining the tensor product of \cstar-algebras.  Given \cstar-algebras $A$ and $B$, their so-called \emph{algebraic tensor product} $A\odot B$ is simply their tensor product as vector spaces.  It is quite easy to see that $A\odot B$ can be equipped with a $*$-algebra structure.  The question becomes:  how can one equip $A\odot B$ with a norm such that its completion with respect to this norm is a \cstar-algebra?  In general, there can be many ways of accomplishing this task, that is, there can be norms on $A\odot B$ whose respective completions are nonisomorphic \cstar-algebras.  For example, a result of Junge and Pisier \cite{jungepisier} is that the algebraic tensor product $B(H)\odot B(H)$ admits a continuum of \cstar-norms whose completions are nonisomorphic \cstar-algebras.

That being said, there are always two canonical norms on $A\odot B$, the so-called \emph{minimal tensor norm} $\|\cdot\|_{\min}$ and \emph{maximal tensor norm} $\|\cdot\|_{\max}$.  As the names indicate, they are indeed the least and greatest norms on $A\odot B$ whose completions are \cstar-algebras.  Their precise definitions will be recalled in the next section, but for now let us remark that they are both defined in terms of representations of the algebras; in other words, their definitions are ``extrinsic.''  One might wonder if it possible to give a more ``intrinsic'' definition of these norms, that is, formulae for these norms that are ``definable'' in terms of the algebra structures on $A$ and $B$ themselves.  To prevent overcomplication, we mainly study this question under the assumption that $A=B$.

Presumably the answer to the previous inquiry, in general, is:  no!  Indeed, if there were such definitions, then they would be written down in every \cstar-algebra textbook.  But the problem with verifying this seemingly obvious attitude is:  what exactly does one mean by ``definable''?  One way of making this idea precise  is to use the language of model theory and that is the approach taken in the current paper.  Roughly speaking, to say that the tensor norm is definable is to say that it can be given by a formula in the natural first-order language for studying \cstar-algebras.  A more general and precise definition of definability can be found in the next section.

Before describing our main results, we first mention that, for $\|\cdot\|_\alpha$ for $\alpha\in\{\min,\max\}$, there are actually two different questions one can ask concerning their definability:  (1) Is the norm $\|\cdot\|_\alpha$ on $A\odot A$ definable uniformly across all \cstar-algebras $A$?  (2)  Given a particular \cstar-algebra $A$, is the norm $\|\cdot\|_\alpha$ on $A\odot A$ definable?  We refer to these questions as the ``global'' and ''local'' questions respectively.

We begin by discussing our results in the global case.  For each of the two choices for $\alpha$, the answer is indeed negative:

\begin{theorem}
For $\alpha\in\{\min,\max\}$, the norm $\|\cdot\|_\alpha$ is not uniformly definable across all \cstar-algebras.
\end{theorem}

To prove this theorem, we introduce the following notion:  given a \cstar-algebra $A$, an ultrafilter $\u$, and $\alpha\in\{\min,\max\}$, we say that $A$ has property $\alpha$-$\u$ if there is a natural isometric embedding $A^\u\otimes_\alpha A^\u\hookrightarrow (A\otimes_\alpha A)^\u$.  (A more precise version of this definition will be given in Section 3 below.)  We show that if the norm $\|\cdot\|_\alpha$ is uniformly definable across all \cstar-algebras, then each \cstar-algebra $A$ has property $\alpha$-$\u$.  On the other hand, we prove the following theorem:

\begin{theorem}
For each $\alpha\in \{\min,\max\}$, there are \cstar-algebras without property $\alpha$-$\u$.
\end{theorem}

In the case of the minimal tensor product, using a deep theorem of Kirchberg, we show that if $A$ is not exact, then $A$ does not have property $\min$-$\u$ for any nonprincipal ultrafilter $\u$ on $\mathbb{N}$.  Meanwhile, the examples of \cstar-algebras without property $\max$-$\u$ use Pisier's recent ultraproduct characterization of the lifting property \cite{pisier}.

On the other hand, if one restricts to $n$-subhomogeneous \cstar-algebras for a fixed $n$, then the minimal tensor product is uniformly definable:

\begin{theorem}
For a fixed positive integer $n$, $\|\cdot\|_{\min}$ is uniformly definable over all $n$-subhomogeneous \cstar-algebras.
\end{theorem}

The proof of this theorem uses, amongst other things, the Beth Definability Theorem from model theory together with the fact that all separable subhomogeneous \cstar-algebras are subalgebras of matrix algebras over commutative \cstar-algebras.

We now turn to the ``local'' situation.  Unfortunately, our results in this case are only partial as showing that $\|\cdot\|_\alpha$ is not definable when restricted to a particular \cstar-algebra seems to be a much more delicate problem.  In order to explain our partial results, say that $\|\cdot\|_\alpha$ is \emph{explicitly definable} for a given \cstar-algebra $A$ if there is an ``effective'' mapping $(n,\epsilon)\mapsto \varphi_{n,\epsilon}(\vec x,\vec y)$, where $\varphi_{n,\epsilon}$ is a formula in $2n$-variables, for which, given any $\vec a,\vec b\in A^n$, one has 
$$\left|\left\|\sum_{k=1}^n a_k\otimes b_k\right\|_{A
\otimes_\alpha A}- \varphi_{n,\epsilon}(\vec a,\vec b)^A\right|<\epsilon.$$ 

Given this definition, we can now state our local results:

\begin{theorem}
Suppose that $A$ is a \emph{locally universal} \cstar-algebra.  Then $\|\cdot\|_{\min}$ is not explicitly definable in $A$.  If, in addition, $A$ has Kirchberg's QWEP property, then $\|\cdot\|_{\max}$ is not explicitly definable in $A$.
\end{theorem}

Here, a \cstar-algebra is locally universal if every \cstar-algebra embeds in an ultrapower of $A$.  We note that every \emph{existentially closed} \cstar-algebra (see \cite{KEP}) is both locally universal and has the QWEP (in fact, the WEP), whence satisfies the hypotheses of the above theorem.  We note also that the hypothesis of the second part of the above theorem can be weakened to assuming that $(A,A)$ is a Tsirelson pair in the sense introduced by the first author and Hart in \cite{tsirelson}.

The previous theorem was motivated by \emph{Kirchberg's embedding problem (KEP)}, which asks if the Cuntz algebra $\mathcal{O}_2$ is locally universal.  Thus, if one can show that $\|\cdot\|_{\min}$ (which coincides with $\|\cdot\|_{\max}$ since $\mathcal{O}_2$ is nuclear) is explicitly definable in $\mathcal{O}_2$, then one can give a negative solution to the Kirchberg embedding problem.

The previous theorem is established using the landmark result in quantum complexity theory known as MIP*=RE \cite{MIP*}.  The proof will actually show that, in the above definition, we only need the defining property of $\|\cdot\|_{\min}$ or $\|\cdot\|_{\max}$ being explicitly definable to hold true for $\epsilon=1/8$.

We can give one further example of a class of \cstar-algebras for which $\|\cdot\|_{\max}$ is not explicitly definable.  The theorem uses the notion of a \emph{qc-full} \cstar-algebra, which will be defined in Section 4 below using notions from quantum complexity theory.

\begin{theorem}
Suppose that $A$ is a qc-full \cstar algebra that $A$ admits a \emph{computable presentation}.  Then $\|\cdot\|_{\max}$ is not explicitly definable in $A$.  
\end{theorem}

Here, when we say that $A$ has a computable presentation, we mean that, roughly speaking, there is some countable sequence $(a_n)_{n\in \mathbb{N}}$ from $A$ that generates a dense $*$-subalgebra of $A$ and for which the norm on $*$-polynomials of the $a_n$'s is effectively approximable; a more precise definition can be found in Section 5.  

An example of a \cstar-algebra satisfying both assumptions of the previous theorem is $C^*(\mathbb{F}_n)$, for $n\in \{2,3,\ldots,\infty\}$, whence we obtain:

\begin{corollary}
The norm $\|\cdot\|_{\max}$ is not explicitly definable in $C^*(\mathbb{F}_n)$.
\end{corollary}

The proof of the previous theorem uses the more recent quantum complexity result MIP$^{co}$=coRE \cite{MIPco}.  The same argument can also be used to establish the following corollary:

\begin{corollary}
For any $n\in \{2,3,\ldots,\infty\}$, the ``standard presentation'' of $C^*(\mathbb{F}_n\times \mathbb{F}_n)$ is not computable.
\end{corollary}

The case $n=2$ of the previous corollary settles a question of Fritz, Netzer, and Thom \cite{canyou}.

We assume that the reader is familiar with basic \cstar-algebra theory.  We also assume that the reader is familiar with basic continuous model theory as it applies to \cstar-algebras; the reader not familiar with this topic can consult \cite{munster} or \cite{bradd}.  That being said, some of the more pertinent results on definability theory will be discussed in the next section.

\section{Preliminaries}

\subsection{Tensor products}

Throughout, fix \cstar-algebras $A$ and $B$.  We denote their algebraic tensor product (that is, their tensor product as vector spaces) by $A\odot B$.  There is a natural $*$-algebra structure on $A\odot B$.  A \textbf{\cstar-norm} on $A\odot B$ is a norm $\|\cdot\|_\alpha$ satisfying the following properties for all $x,y\in A\odot B$:
\begin{itemize}
    \item $\|xy\|_\alpha\leq \|x\|_\alpha\|y\|_\alpha$;
    \item $\|x^*\|_\alpha=\|x\|_\alpha$;
    \item $\|x^*x\|_\alpha=\|x\|_\alpha^2$.
\end{itemize}
The completion of $A\odot B$ with respect to $\|\cdot\|_\alpha$ is then a \cstar-algebra, denoted $A\otimes_\alpha B$.
In this paper, we focus on two specific \cstar-norms on tensor products:  the minimal tensor product norm and the maximal tensor product norm.

Given faithful representations $A\subseteq B(H_A)$ and $B\subseteq B(H_B)$, we may faithfully represent $A\odot B\subseteq B(H_A\otimes H_B)$; the \textbf{minimal tensor product norm} $\|\cdot\|_{\min}$ on $A\odot B$ is defined to be the norm it inherits as a $*$-subalgebra of $B(H_A\otimes H_B)$.  One must check that this norm is independent of faithful representation.  A deep theorem of Takesaki states that $\|\cdot\|_{\min}\leq \|\cdot\|_\alpha$ for any \cstar-norm $\|\cdot\|_\alpha$ on $A\odot B$.  The minimal tensor product is so important in \cstar-algebra theory that it is often denoted simply as $A\otimes B$ rather than $A\otimes_{\min}B$.

The \textbf{maximal tensor product  norm} on $A\odot B$ is defined by
$$\|x\|_{\max}:=\sup\{\|\pi(x)\| \ : \ \pi:A\odot B\to B(H) \text{ a }*\text{-homomorphism}\}.$$  
It is clear from the definition that $\|\cdot\|_\alpha\leq \|\cdot\|_{\max}$ for any \cstar-norm $\|\cdot\|_\alpha$ on $A\odot B$.  It is also clear from the definition that the maximal tensor product satisfies the following universal property:  given a \cstar-algebra $C$ and $*$-homomorphisms $\pi_A:A\to C$ and $\pi_B:B\to C$ such that $[\pi_A(a),\pi_B(b)]=0$ for all $a\in A$ and $b\in B$, there is a unique $*$-homomorphism $\pi:A\otimes_{\max}B\to C$ extending the $*$-homomorphism $\pi_A\odot \pi_B:A\odot B\to C$.

The minimal tensor product enjoys the following property:  if $A\subseteq B$, then for any \cstar-algebra $C$, the inclusion of $*$-algebras $A\odot C\subseteq B\odot C$ extends to an isometric inclusion $A\otimes_{\min} B\subseteq B\otimes_{\min} C$.  In the same context, by the universal mapping property, there is always a $*$-homomorphism $A\otimes_{\max}C\to B\otimes_{\max} C$; in general, this mapping usually fails to be isometric.  We say that the inclusion $A\subseteq B$ is \textbf{relatively weakly injective (r.w.i.)} if, for any \cstar-algebra $C$, the canonical map $A\otimes_{\max}C\to B\otimes_{\max}C$ is isometric.  A \cstar-algebra that is r.w.i. in every \cstar-algebra containing it is said to have Lance's \textbf{weak expectation property (WEP)}.  A \cstar-algebra that is a quotient of a \cstar-algebra with the WEP is said to have Kirchberg's \textbf{QWEP property}.

\subsection{Model-theoretic definability}

The following is a basic summary of the notion of definability in continuous logic; a much more detailed presentation can be found in \cite{spgap}.

Recall that the formulae in the language of \cstar-algebras are formed by first considering the atomic formulae $\|p(\vec x)\|$, where $p$ is a $*$-polynomial (and, for simplicty, $\vec x$ are assumed to range over the operator norm ball), and then closing under continuous combinations (that is, if $\varphi_1,\ldots,\varphi_n$ are formulae and $u:\mathbb{R}^n\to \mathbb{R}$ is a continuous function, then $u(\varphi_1,\ldots,\varphi_n)$ is once again a formula) and the quantifiers $\sup_x$ and $\inf_x$ (where, again, the variable is assumed to range over the operator norm unit ball).

Now fix a theory $T$ extending the theory of \cstar-algebras.  We can use $T$ to define a pseudometric on the collection of formulae in the variables $x_1,\ldots,x_n$ given by 
$$d_T(\varphi_1,\varphi_2):=\sup\{|\varphi_1(\vec a)^A-\varphi_2(\vec a)^A| \ : \ A\models T\}.$$  The elements of the separation-completion of this pseudometric space are called \textbf{$T$-formulae}.  Given a $T$-formula $\varphi(\vec x)$ and a model $A\models T$, we get a well-defined interpretation $\varphi^A:A_1^n\to \mathbb{R}$, where $A_1$ denotes the operator norm unit ball of $A$.  Note that the $\varphi^A$'s are uniformly continuous with modulus of uniform continuity independent of the model $A$.

Conversely, suppose that we have, for each model $A$ of $T$, an $n$-ary predicate $P^A:A_1^n\to [0,1]$; we refer to the assignment $A\mapsto P^A$ as a \textbf{$T$-function} (we may also simply say that $P$ is a $T$-function).  We say that a $T$-function is \textbf{realized} if it is the result of interpreting a $T$-formula.

The following will be important for us in the sequel:

\begin{prop}\label{functionfact}
Suppose that $P$ is a $T$-function for which each $P^A$ is uniformly continuous with uniform continuity modulus independent of the model $A$.  Then $P$ is realized if and only if:
\begin{enumerate}
    \item Whenever $A,B\models T$ are isomorphic, then any isomorphism $\Phi:A\to B$ satisfies $P^B\circ \Phi=P^A$.
    \item Whenever $(A_i)_{i\in I}$ is a family of models of $T$ and $\u$ is an ultrafilter on $I$, then denoting $A:=\prod_\u A_i$, we have $P^A=\lim_\u P^{A_i}$.
    %\item Whenever $A\models T$ and $Q:A\to \mathbb{R}$ is such that $\lim_\u Q=\lim_\u P^A$, then $Q=P^A$.
\end{enumerate}
\end{prop}

\begin{proof}
It is clear that if $P$ is realized, then the above items hold.  Conversely, suppose that the above items hold.  Let $\mathcal{C}:=\{(A,P^A) \ : \ A\models T\}$, a class of structures in the language of \cstar-algebras expanded by a new predicate symbol.  By an application of the Beth Definability Theorem (\cite[Corollary 4.2.2]{munster}), it suffices to show that $\mathcal{C}$ is an elementary class, which, by \cite[Theorem 2.4.1]{munster}, is equivalent to showing that $\mathcal{C}$ is closed under isomorphisms, ultraproducts, and ultraroots.  

First suppose that $(A,P^A)\in \mathcal{C}$ is isomorphic to $(B,Q)$; we want to prove that $(B,Q)$ belongs to $\mathcal{C}$, that is, that $Q=P^B$.  But if the isomorphism is given by $\Phi:(A,P^A)\to (B,Q)$, then on the one hand $P^A=Q\circ \Phi$, but by assumption (1), we also have $P^A=P^B\circ \Phi$; it follows that $Q=P^B$, as desired.  For closure under ultraproducts, if $A=\prod_\u A_i$, then $$\prod_\u (A_i,P^{A_i})=(A,\lim_\u P^{A_i})=(A,P^{A})\in \mathcal{C},$$ where the second equality follows by item (2) above.  Finally, if $(A,Q)$ is a structure for which $(A,Q)^\u\in \mathcal{C}$, then by (2) again, we have $\lim_\u Q=P^{A^\u}=\lim_\u P^A$; since $Q=(\lim_\u Q)|A$ and $P^A=(\lim_\u P^A)|A$, we have that $P^A=Q$ and thus $(A,Q)=(A,P^A)\in \mathcal{C}$.  
\end{proof}

\section{The global case}

\subsection{The minimal tensor norm}

For a \cstar-algebra $A$, we define the predicate $P^{\min}_{A,n}:A_1^{2n}\to [0,2n]$ by $P^{\min}_{A,n}(\vec a,\vec b):=\|\sum_{k=1}^n a_k\otimes b_k\|_{A\otimes A}$.  Given a theory $T$ of \cstar-algebras, we can consider the $T$-functions $P^{\min}_n$.\footnote{Following the notation from the previous section, we should probably have written $(P^{\min}_n)^A$ instead of $P^{\min}_{A,n}$, but we felt that the latter notation looked nicer in print.}  We say that $P^{\min}$ is \textbf{$T$-definable} if the $T$-functions $P^{\min}_{n}$ are realized for each $n$.

\begin{prop}\label{mincriterion}
If $T$ is a theory of \cstar-algebras, then $P^{\min}$ is $T$-definable if and only if:  for all families $(A_i)_{i\in I}$ of models of $T$, all ultrafilters $\u$ on $I$, and all $n$, we have $$P^{\min}_{\prod_{\u}A_i,n}=\lim_\u P^{\min}_{A_i,n}.$$  In this case, each $P^{\min}_n$ is $T$-equivalent to a quantifier-free $T$-formula.\footnote{A quantifier-free $T$-formula is one that can be realized as a limit, with respect to the metric $d_T$ above, of quantifier-free formulae, that is, formulae constructed without the use of the quantifers $\sup$ and $\inf$.}
\end{prop}

\begin{proof}
The first statement follows from Proposition \ref{functionfact} above; note that, in the current context, item (1) of that proposition holds automatically as any isomorphism $A\to B$ extends to an isomorphism $A\otimes A\to B\otimes B$.  To see the second statement, we use the usual test for a $T$-formula to be equivalent to a quantifier-free $T$ formula \cite[Proposition 13.2]{mtfms}.  Suppose that the $T$-function $P^{\min}_n$ is realized by the $T$-formula $\varphi(\vec x,\vec y)$.  If $A,B\models T$ and $C$ is a common \cstar-subalgebra of $A$ and $B$ (which need not be a model of $T$), then for all $\vec c,\vec d\in C^n$, we have:
$$\varphi^{A}(\vec c,\vec d)=P^{\min}_{A,n}(\vec c,\vec d)=P^{\min}_{C,n}(\vec c,\vec d)=P^{\min}_{B,n}(\vec c,\vec d)=\varphi^B(\vec c,\vec d).$$  The second equality comes from the fact that $C\otimes C\subseteq A\otimes A$ and similarly for the third equality.  It follows that $\varphi$ is $T$-equivalent to a quantifier-free $T$-formula.
\end{proof}

Fix a \cstar-algebra $A$ and an ultrafilter $\u$.  For $k=1,\ldots,n$, consider elements $a^k=(a^k_i)_\u,b=(b^k_i)_\u\in A^\u$; we can view the element $\sum_{k=1}^n a^k\otimes b^k$ of $A^\u\odot A^\u$ as an element of $(A\otimes A)^\u$, namely as the element $(\sum_{k=1}^n a^k_i\otimes b^k_i)_\u$; this leads to a map $\iota_{A,\u}^{\min}:A^\u\odot A^\u \hookrightarrow (A\otimes A)^\u$.  We say that $A$ has \textbf{property min-$\u$} if the norm on $A^\u\odot A^\u$ induced by $\iota_{A,\u}^{\min}$ is the minimal tensor norm, that is, if $\iota_{A,\u}$ extends to an isometric embedding $\iota_{A,\u}^{\min}:A^\u\otimes A^\u\to (A\otimes A)^\u$.

\begin{prop}\label{minsubalgebra}
Given a \cstar-algebra $A$, we have that $A$ has property min-$\u$ if and only if $P^{\min}_{A^\u,n}=\lim_\u P^{\min}_{A,n}$ for all $n$.
\end{prop}

\begin{proof}
The proposition follows from the observations that $$\left\|\sum_{k=1}^n a_k\otimes b_k\right\|=P_{A^\u,n}(\vec a,\vec b)$$ and $$\left\|\iota^{\min}_{A,\u}\left(\sum_{k=1}^n a_k\otimes b_k\right)\right\|=\lim_\u \left\|\sum_{k=1}^n a_i^k\otimes b_i^k\right\|=\lim_\u P_{A,n}(\vec a_i,\vec b_i).$$
\end{proof}

\begin{cor}
If $T$ is a theory of \cstar-algebras and $P^{\min}$ is $T$-definable, then for every model $A\models T$ and every ultrafilter $\u$, $A$ has property min-$\u$.
\end{cor}

\begin{question}
If $A$ has property min-$\u$ for some nonprincipal ultrafilter $\u$, must it have property min-$\mathcal{V}$ for all nonprincipal ultrafilters $\mathcal{V}$?
\end{question}

\begin{prop}
Suppose that $A\subseteq B$.   If $B$ has property min-$\u$, then $A$ has property min-$\u$.
% \begin{enumerate}
%     \item
%     \item If $P^{\min}$ is $\Th(B)$-definable, then $P^{\min}$ is $\Th(A)$-definable.
% \end{enumerate}
\end{prop}

\begin{proof}
For each $n$, we have that $$P^{\min}_{A^\u,n}=P^{\min}_{B^\u,n}|(A^\u)^{2n}=(\lim_\u P^{\min}_{B,n})|(A^\u)^{2n}=\lim_\u P^{\min}_{A,n}.$$

% (2) Fix $n$ and let $\varphi(\vec x,\vec y)$ be a quantifier-free $\Th(B)$-formula that agrees with $P^{\min}_n$ in models of $\Th(B)$.  We claim $\varphi(\vec x,\vec y)$ is also a $\Th(A)$-formula and that it agrees with $P^{\min}_n$ in models of $\Th(A)$.  The former follows from the fact that $\varphi(\vec x,\vec y)$ is quantifier-free.  To see the second assertion, suppose that $A'\equiv A$.  Then there is $B'\equiv B$ such that $A'\subseteq B'$.  But then, for $\vec a,\vec b\in A'$, we have
% $$P^{\min}_{A',n}(\vec a,\vec b)=P^{\min}_{B',n}(\vec a,\vec b)=\varphi^{B'}(\vec a,\vec b)=\varphi^{A'}(\vec a,\vec b),$$
% where the last equality follows from the fact that $\varphi$ is quantifier-free.
\end{proof}

\begin{prop}\label{minexact}
Suppose that $A$ has property min-$\u$ for some nonprincipal ultrafilter $\u$ on $\mathbb{N}$.  Then $A$ is exact.
\end{prop}

\begin{proof}
Set $M:=\prod_\u M_n(\mathbb{C})$.  By a theorem of Kirchberg \cite{kirchberg}, it suffices to show that the norm on $A\odot M$ induced by the embedding $A\odot M\subseteq \prod_\u M_n(A)$ is the minimal tensor norm.  Fix $x:=\sum_{k=1}^m a^k\otimes m^k\in A\odot M$.  Without loss of generality, we may assume that $A$ is not subhomogeneous.  Consequently, there are ucp maps $\psi_n:M_n(\mathbb C)\to A$ such that $$\psi:=(\psi_n)_\u:M\to A^\u$$ is a ucp embedding with conditional expectation, or, more precisely, there are ucp maps $\rho_n:A\to M_n(\mathbb{C})$ such that, setting $\rho:=(\rho_n)_\u:A^\u\to M$, we have that $\rho\circ \psi=\id_{M}$ (see \cite[Lemma 4.16]{omitting}).  We then have that
$$\|x\|_{A\otimes M}=\|(\id_{A^\u}\otimes \psi)(x)\|_{A^\u\otimes A^\u}=\left\|\sum_k (a^k\otimes \psi_n(m^k_n))_\u\right\|_{(A\otimes A)^\u},$$ where the last equality holds since $A$ has property min-$\u$.  In other words, $\|x\|_{A\otimes M}=\|\theta(x)\|_{(A\otimes A)^\u}$, where $\theta:=(\id_A\otimes \psi_n)_\u:\prod_\u M_n(A)\to (A\otimes A)^\u$.  To conclude the proof, it suffices to show that $\theta$ is norm-preserving.  Set $\omega:=(\id_A\otimes \rho_n)_\u:(A\otimes A)^\u\to \prod_\u M_n(A)$; it suffices to show that $\omega\circ \theta=\id_{\prod_{\u} M_n(A)}$.  To see this, given $\sum_{ij}a^{ij}\otimes e_{ij}\in \prod_\u A\otimes M_n(\mathbb{C})$, we have that 
$$\omega\left(\theta\left(\sum_{ij}a^{ij}\otimes e_{ij}\right)\right)=\omega\left(\sum_{ij}(a^{ij}_n\otimes \psi_n(e_{ij}))_\u\right)=\sum_{ij}(a_n^{ij}\otimes \rho_n(\psi_n(e_{ij}))_\u;$$ since $\lim_\u \rho_n(\psi_n(e_{ij}))=e_{ij}$ and $\lim_\u \|a_n^{ij}\|<\infty$ for all $i,j$, this finishes the proof.
\end{proof}

We recall the following fact (see \cite[Proposition 3.14.1]{munster}):

\begin{fact}\label{subhomfact}
For a \cstar-algebra $A$, we have that $A$ is subhomogeneous if and only if every \cstar-algebra elementarily equivalent to $A$ is exact.
\end{fact}

We say that a theory $T$ of \cstar-algebras is \textbf{subhomogeneous} if all models of $T$ are subhomogeneous.  In this case, all models of $T$ must be $n$-subhomogeneous for some $n$; indeed, if this were not the case, then there would be a family $(A_n)_{n\in \mathbb{N}}$ of models of $T$ such that $A_n$ is not $n$-subhomogeneous; it follows that $\prod_\u A_n$ is a model of $T$ that is not subhomogeneous (see the proof of \cite[Proposition 3.14.1]{munster}).  By Fact \ref{subhomfact}, a theory $T$ is subhomogeneous if and only if all models of $T$ are exact.  Proposition \ref{minexact} thus implies:

\begin{cor}
If $T$ is a theory of \cstar-algebras for which $P^{\min}$ is $T$-definable, then $T$ is subhomogeneous. 
\end{cor}

We now aim to prove the converse of the previous corollary.
% \begin{fact}
% There is a quantifier-free $T_{C^*}$-formula $Q_n(\vec x)$ such that, in all \cstar-algebras $A$, we have that $Q(\vec a)=\|(a_{ij})\|$.
% \end{fact}
% \begin{proof}
% That such a formula exists was proven by Lupini in \cite{}.  The fact that the formula must be quantifier-free follows from the same sort of analysis as in ??? above.
% \end{proof}
In the next proposition, we let $T_{\ab}$ denote the theory of abelian \cstar-algebras and we let $T_{\ab}'$ denote the theory of pairs $(A,B)$ of \cstar-algebras such that $B$ is abelian.  We abuse notation and let $P^{\min}_n$ denote the $T'_{\ab}$-function given by $$P^{\min}_{A,B,n}(\vec a,\vec b):=\left\|\sum_{k=1}^n a_k\otimes b_k\right\|_{A\otimes B}.$$  We similarly speak of $P^{\min}$ being $T'_{\ab}$-definable.

\begin{prop}
$P^{\min}$ is both $T_{\ab}'$-definable and $T_{\ab}$-definable.
%The predicate $P_{A,B,n}^{\min}$ is uniformly definable over all pairs $(A,B)$ such that $B$ is abelian.  In particular, $P^{\min}$ is $T_{ab}$-definable.
\end{prop}

\begin{proof}
The second statement clearly follows from the first.  The key observation in proving the first statement is that, given a \cstar-algebra $A$ and a compact Hausdorff space $X$, we have that $A\otimes C(X)\cong C(X,A)$ via the map which sends $\sum_{i=1}^n a_i\otimes f_i$ to the function $x\mapsto \sum_{i=1}^n f_i(x)a_i$.  It follows that
$$\left\|\sum_{i=1}^n f_i\otimes a_i\right\|=\sup_{x\in X}\left\|\sum_{i=1}^n f_i(x)a_i\right\|.$$  But the set $\{(f_i(x))_{i=1}^n: \ x\in X\}$ correspond precisely to the joint spectrum $\sigma(f)$ of $f=(f_i)_{i=1}^n$, whence $$\left\|\sum_{i=1}^n f_i\otimes a_i\right\|=\sup_{\lambda\in \sigma(f)}\left\|\sum_{i=1}^n \lambda_ia_i\right\|.$$  Now the joint spectrum is a definable family of definable sets (in the sense of \cite[Definition 5.8]{randomization}) relative to the theory $T_{\ab}$ as can be verified from the proof of the fact that the set of pairs $\{(f,\lambda) \ : \ \lambda\in \sigma(f)\}$ is definable given in \cite[Proposition 5.21]{eaglevignati}.  By \cite[Proposition 5.9]{randomization}, for any theory $T$, a $T$-function obtained from a realized $T$-formula by quantifying over a definable family of definable sets is again a realized $T$-formula.  The conclusion of the theorem follows.
\end{proof}

\begin{prop}\label{homogeneousmin}
For each $N\geq 1$, each compact Hausdorff space $X$, and each ultrafilter $\u$, we have that $M_N(C(X))$ has property min-$\u$.
% $M_N(C(X))^\u\otimes M_N(C(X))^\u\subseteq (M_N(C(X))\otimes M_N(C(X)))^\u$.  (Probably even $P^{\min}$ is $\Th(M_N(C(X))$-definable; are all models of that theory of the form $M_N(A)$ for $A$ abelian?)
\end{prop}

\begin{proof}
We simply calculate:
\begin{alignat}{2}
(M_N(C(X)))^\u\otimes (M_N(C(X)))^\u &\cong M_N(C(X)^\u)\otimes M_N(C(X)^\u)\notag \\ \notag
                        &\cong M_{N^2}(C(X)^\u\otimes C(X)^\u) \\ \notag
                                 &\subseteq M_{N^2}((C(X)\otimes C(X))^\u) \\ \notag
                                 &\cong (M_{N^2}\otimes C(X)\otimes C(X))^\u\\ \notag 
                                 &\cong (M_N(C(X))\otimes M_N(C(X))^\u. \notag
\end{alignat}
Notice that the isometric inclusion above used the fact that $C(X)$ has property min-$\u$, which follows from the fact that $P^{\min}$ is $T_{\ab}$-definable.
\end{proof}

\begin{thm}
Suppose that $T$ is the theory of $n$-subhomogeneous \cstar-algebras.  Then $P^{\min}$ is $T$-definable.
\end{thm}

\begin{proof}
% Let $\varphi(\vec x,\vec y)$ be the quantifier-free formula defining $P^{\min}$ in $\Th(M_N(C(2^\omega)))$.  It suffices to show that $\varphi$ defines $P^{\min}$ in $\Th(A)$ for any $n$-subhomogeneous \cstar-algebra $A$.  But if $A'\equiv A$, then $A'$ embeds into an ultrapower of $M_N(C(2^\omega))$ by a result in Blackadar's book.  Now we can use the previous proposition.
By Proposition \ref{mincriterion}, it suffices to show that if $(A_i)_{i\in I}$ is a family of $n$-subhomogeneous \cstar-algebras, $\u$ is a non-principal ultrafilter on $I$, and $m\geq 1$, then $P^{\min}_{\prod_{\u}A_i,m}=\lim_\u P^{\min}_{A_i,m}$.  We first suppose that each $A_i$ is separable.  We may thus assume that there is an integer $N$ such that each $A_i\subseteq B:=M_N(C(2^\omega))$ (see \cite[IV.1.4.4]{blackadar}).  Setting $A:=\prod_\u A_i$, it follows that
$$P_{A,m}^{\min}(\vec a,\vec b)=P_{B^\u,m}^{\min}(\vec a,\vec b)=\lim_\u P_{B,m}(\vec a_i,\vec b_i)=\lim_\u P_{A_i,m}(\vec a,\vec b_i),$$
where the second equality uses Proposition \ref{homogeneousmin} above.

In the general case, fix $\vec a=(\vec a_i)_\u,\vec b=(\vec b_i)_\u\in A:=\prod_\u A_i$.  For each $i\in I$, let $A_i'\subseteq A_i$ be a separable subalgebra such that $\vec a_i,\vec b_i\in A_i'$.  Set $A':=\prod_\u A_i'$, so $\vec a,\vec b\in A'$.  Then
$$P_{A,m}^{\min}(\vec a,\vec b)=P_{A',m}^{\min}(\vec a,\vec b)=\lim_\u P^{\min}_{A_i',m}(\vec a,\vec b)=\lim_\u P^{\min}_{A_i,m}(\vec a,\vec b),$$
where the second equality uses the fact that we established the claim when all algebras are separable.
\end{proof}

Summarizing, we now have:

\begin{cor}
If $T$ is a theory of \cstar-algebras, then $P^{\min}$ is $T$-definable if and only if $T$ is subhomogeneous.    
\end{cor}

% \begin{proof}
% One direction follows from ??? above.  In the other direction, if for each $n$, there is a non-$n$-subhomogeneous model, then the ultraproduct of those models is a model of $T$ that is not subhomogeneous.
% \end{proof}

\subsection{The maximal tensor norm}

Fix a \cstar-algebra $A$.  We define the predicate $P^{\max}_{A,n}:A_1^{2n}\to [0,2n]$ by $P^{\max}_{A,n}(\vec a,\vec b):=\|\sum_{k=1}^n a_k\otimes b_k\|_{A\otimes_{\max} A}$.  Given a theory $T$ of \cstar-algebras, we say that $P^{\max}$ is \textbf{$T$-definable} if the $T$-functions $P^{\max}_{n}$ are realized for each $n$.

\begin{thm}
If $T$ is a theory of \cstar-algebras, then $P^{\max}$ is $T$-definable if and only if:  for all families $(A_i)_{i\in I}$ of models of $T$ and all ultrafilters $\u$ on $I$, we have $P^{\max}_{\prod_{\u}A_i,n}=\lim_\u P^{\max}_{A_i,n}$ for all $n$.  In this case, each $P^{\min}_n$ is $T$-equivalent to an existential $T$-formula.\footnote{An existential $T$-formula is one that can be realized as a limit, with respect to the metric $d_T$ above, of existential formulae, where a formula is existential if it can be written as a string of existential quantifiers followed by a quantifier-free formula.}
\end{thm}

\begin{proof}
The only part that is different from the case of the minimal tensor norm is the last part.  Here, given models $A,B\models T$ with $A\subseteq B$, the universal property of maximal tensor products ensures that we have a $*$-homomorphism $A\otimes_{\max} A\to B\otimes_{\max} B$, whence, for any $\vec a,\vec b\in A^n$, we have $P^{\max}_{A,n}(\vec a,\vec b)\geq P^{\min}_{B,n}(\vec a,\vec b)$.  The result now follows from \cite[Proposition 2.4.6]{munster}.
\end{proof}

Fix a \cstar-algebra $A$ and an ultrafilter $\u$.  By the universal mapping property of maximal tensor products, there is a unique *-homomorphism $$\iota_{A,\u}^{\max}:A^\u\otimes_{\max}A^\u\to (A\otimes_{\max}A)^\u$$ for which $\iota_{A,\u}^{\max}(\sum_{k=1}^n a^k\otimes b^k)=(\sum_{k=1}^n a^k_i\otimes b^k_i)_\u$.  We say that $A$ has property \textbf{max-$\u$} if $\iota_{A,\u}^{\max}$ is an isometric embedding.  As in the case of minimal tensor products, we have the following:

\begin{prop}
Given a \cstar-algebra $A$, we have that $A$ has property max-$\u$ if and only if $P^{\max}_{A^\u,n}=\lim_\u P^{\max}_{A,n}$ for all $n$.
\end{prop}

\begin{cor}
If $T$ is a theory of \cstar-algebras and $P^{\max}$ is $T$-definable, then for every model $A\models T$ and every ultrafilter $\u$, $A$ has property max-$\u$.
\end{cor}

\begin{question}
If $A$ has property max-$\u$ for some nonprincipal ultrafilter $\u$, must it have property max-$\mathcal{V}$ for all nonprincipal ultrafilters $\mathcal{V}$?
\end{question}

We now aim to show that there are \cstar-algebras without property max-$\u$ for some $\u$ and thus that $P^{\max}$ is not definable with respect to the theory of all \cstar-algebras.  We begin by explaining some aspects of Pisier's ultraproduct characterization of the lifting property (LP) for \cstar-algebras \cite{pisier}.

First, recall that a \cstar-algebra $A$ has the \textbf{lifting property (LP)} if every ucp map $A\to C/J$ from $A$ into a quotient \cstar-algebra has a ucp lift $A\to C$.

Next, recall that, given any family $(D_i)_{i\in I}$ of \cstar-algebras, an ultrafilter $\u$ on $I$, and another \cstar-algebra $A$, there is a natural map $$\left(\prod_\u D_i\right)\otimes_{\max}A\to \prod_\u (D_i\otimes_{\max}A).$$

The following is \cite[Theorem 6.1]{pisier}:

\begin{fact}\label{pisierfact}
A \cstar-algebra $A$ has the LP if and only if, for any family $(D_i)_{i\in I}$ of \cstar-algebras and ultrafilter $\u$ on $I$, the natural map $$\left(\prod_\u D_i\right)\otimes_{\max}A\to \prod_\u (D_i\otimes_{\max}A)$$ is an isometric embedding.
\end{fact}

For what follows, we will need a more precise version of the previous fact.  We say a \cstar-algebra $A$ is \textbf{surjectively universal} if, for all separable \cstar-algebras $B$, there is a surjective $*$-homomorphism $A\to B$.  For example, $C^*(\mathbb{F}_\infty)$ is surjectively universal.  Note that in this definition we do not require that $A$ itself is separable.  Also, given a directed set $(I,\leq)$, we say a nonprincipal ultrafilter $\u$ on $I$ is \textbf{cofinal} if all sets of the form $\{i\in I \ : \ i\geq i_0\}$ (as $i_0$ ranges over $I$) belong to $\u$. 

\begin{prop}\label{pisierprecise}
Suppose that $(D_n)_{n\in \mathbb{N}}$ is a family of surjectively universal \cstar-algebras.  Let $\hat {\mathbb{N}}$ denote the set of finite subsets of $\mathbb N$, viewed as a directed set under inclusion.  Let $\u$ be a cofinal ultrafilter on $\hat{\mathbb{N}}$.  Finally, suppose that $A$ is a \cstar-algebra for which the natural map $(\prod_\u D_J)\otimes_{\max}A\to \prod_\u (D_J\otimes_{\max}A)$ is an isometric embedding, where $D_J:=\ell^\infty(D_n \ : \ n\in J)$.  Then $A$ has the LP.
\end{prop}

\begin{proof}
The proof of \cite[Theorem 6.1]{pisier} shows that the assumption of the proposition implies that the family $(D_n)_{n\in \mathbb{N}}$ and algebra $A$ satisfy an inequality labeled (0.1).  Moreover, \cite[Corollary 10.6]{pisier} shows that when $A$ satisfies (0.1) for a family of surjectively universal \cstar-algebras indexed by $\mathbb{N}$, then it holds for all families (this is shown when each $D_n=C^*(\mathbb{F}_\infty)$, but all that is actually used about $C^*(\mathbb{F}_\infty)$ is its surjective universality); and \cite[Theorem 7.5]{pisier} shows that satisfying (0.1) for all families is equivalent to the LP.
% Moreover, the proof of Corollary 10.6 shows that the surjective universality of the family $(D_n)_{n\in I}$ shows that (0.1) holds with this specific family $(D_i)_{i\in I}$ of surjectively universal \cstar-algebras replaced by \emph{any} family $(D_i)_{i\in I}$ of separable \cstar-algebras.  Moreover, as explained in Remark 0.6, this latter fact holds even if the $D_i$'s are not assumed to be separable.  Finally, in Theorem 7.5, it is shown that for an arbitrary \cstar-algebra $A$, satisfying (0.1) for all families $(D_i)_{i\in I}$ implies that $A$ has the lifting property.  
\end{proof}

\begin{prop}
Suppose that $A$ is a \cstar-algebra with the following three properties:
\begin{enumerate}
    \item $A$ is surjectively universal.
    \item $A\cong A\times A$.
    \item $A$ does not have the LP.
\end{enumerate}
Then for any cofinal ultrafilter $\u$ on $\hat{\mathbb{N}}$, we have that $A$ does not have property max-$\u$.  In particular, $P^{\max}$ is not $\Th(A)$-definable.
\end{prop}

\begin{proof}
By (2), for each $J\in \hat{\mathbb{N}}$, we may fix an isomorphism $\psi_J:A_J\to A$.  If $A$ has property max-$\u$, we then have that: 
\begin{alignat}{2}
\left(\prod_\u A_J\right)\otimes_{\max} A &\cong A^\u\otimes_{\max}A\notag \\ \notag
                    &\subseteq A^\u\otimes_{\max}A^\u\\ \notag
                    &\subseteq (A\otimes_{\max} A)^\u\\ \notag
                    &\cong\prod_\u (A_J\otimes_{\max}A). \notag
\end{alignat}
Here, the first isomorphism is induced by $(\psi_J)_\u\otimes id_A$ while the last isomorphism is induced by $(\psi_J^{-1}\otimes id_A)_\u$.  Also, the first isometric inclusion uses that $A$ is r.w.i. in $A^\u$ (see \cite[Proposition 2(2)]{modeltheoryQWEP}) while the second isometric inclusion uses the fact that $A$ has property max-$\u$.
By Proposition \ref{pisierprecise}, we have that $A$ has the LP, a contradiction.
\end{proof}

It remains to show that there is a \cstar-algebra with the properties in the hypothesis of the previous proposition: 

\begin{lemma}
There is a surjectively universal \cstar-algebra without the lifting property for which $A\cong A\times A$.
\end{lemma}

\begin{proof}
Suppose that $G$ is a countable discrete group for which $A':=C^*(G)$ does not have the LP (see, for example, \cite[Corollary 5]{ozawanouniversal} or \cite[Section 9]{pisier}).  Note that $C^*(\mathbb{F}_\infty)*A'\cong C^*(\mathbb{F}_\infty*G)$ is surjectively universal.  We claim that $C^*(\mathbb{F}_\infty)*A'$ does not have the LP.  To see this, suppose that $A'\to B/J$ is a ucp map without a ucp lift.  Take any $*$-homomorphism $C^*(\mathbb{F}_\infty)\to B/J$.  By a result of Boca \cite{boca}, we have a ccp map $C^*(\mathbb{F}_\infty)*A'\to B/J$.  If $C^*(\mathbb{F}_\infty)*A'$ had the LP, we would have a ucp lift $C^*(\mathbb{F}_\infty)*A'\to B$, which then restricts to a ucp lift $A'\to B$, yielding a contradiction.

For each $n\in \mathbb N$, set $A_n:=C^*(\mathbb F_\infty)*A'$.  Set $A:=\ell^\infty(A_n \ : \ n\in \mathbb N)$.  Since there is a surjective $*$-homomorphism $A\to C^*(\mathbb{F}_\infty)$, we have that $A$ is surjectively universal.  If $A$ had the LP, then since $C^*(\mathbb F_\infty)*A'$ is separable and r.w.i. in $A$, if $A$ had the lifting property, then so would $C^*(\mathbb F_\infty)*A'$, a contradiction to the above paragraph.  Finally, by design, we have that $A\cong A\times A$.
\end{proof}

The \cstar-algebra constructed in the previous lemma is not separable.  This raises the question:

\begin{question}
Is there a \emph{separable}, surjectively universal \cstar-algebra $A$ without the lifting property for which $A\cong A\times A$?
\end{question}

The previous question notwithstanding, there is a way to use these ideas to get a separable \cstar-algebra $A$ that does not have property $\max-\u$ for any cofinal ultrafilter $\u$ on $\hat{\mathbb{N}}$.  Once again, fix a countable group $G$ such that $C^*(G)$ does not have the LP.  Let $D:=C^*(\mathbb{F}_\infty*G)$.  As above, $D$ is a surjectively universal group without the LP.  For each $n\in \mathbb N$, set $D_n=D$ and set $A:=\ast_J D_J$.  Now $D_J$ is r.w.i. in $A$ as each $D_K$ is surjectively universal.  In particular, $A$ does not have the LP (else $D_0=D$ would have the LP).

We will need the following lemma.  We speculate that this result might be known amongst experts, but our proof uses Fact \ref{pisierfact}, which is rather recent.

\begin{lemma}\label{rwiultra}
Suppose that $(A_i)_{i\in I}$ and $(B_i)_{i\in I}$ are two families of \cstar-algebras for which $A_i\subseteq B_i$ is r.w.i. for each $i\in I$.  Then for any ultrafilter $\u$ on $I$, one has $\prod_\u A_i\subseteq \prod_\u B_i$ is also r.w.i.    
\end{lemma}

\begin{proof}
By \cite[Theorem 3.3]{OzQ}, it suffices to show that the natural map $$\left(\prod_\u A_i\right)\otimes C^*(\mathbb F_\infty)\to \left(\prod_\u B_i\right)\otimes_{\max} C^*(\mathbb{F}_\infty)$$ is isometric.  But this follows from Fact \ref{pisierfact} since $C^*(\mathbb F_\infty)$ has the LP, as the following diagram indicates:
\[
\begin{tikzcd}
(\prod_\u B_i)\otimes_{\max}C^*(\mathbb{F}_\infty) \arrow[r,hook] & \prod_\u (B_i\otimes_{\max}C^*(\mathbb{F}_\infty))\\
(\prod_\u A_i)\otimes_{\max}C^*(\mathbb{F}_\infty)\arrow[u] \arrow[r,hook] &\prod_\u (A_i\otimes_{\max}C^*(\mathbb{F}_\infty)) \arrow[u,hook]
\end{tikzcd}
\]
\end{proof}

Returning to the previous context, we thus have that $\prod_\u D_j\subseteq A^\u$ is r.w.i.  As the following diagram shows, we have that $(\prod_\u D_j)\otimes_{\max} A\subseteq \prod_\u (D_j\otimes_{\max}A)$ is isometric, implying that $A$ has the LP, a contradiction.
\[
\begin{tikzcd}
     A^\u\otimes_{\max} A \arrow[r] & A^\u\otimes_{\max} A^\u \arrow[r, hook] & \prod_\u (A\otimes_{\max}A) \\
    & (\prod_\u D_J)\otimes_{\max}A \arrow[u, hook] \arrow[r] & \prod_\u (D_J\otimes_{\max}A)\arrow[u, hook]
\end{tikzcd}
\]

Summarizing this discussion, we have the following:

\begin{thm}
Suppose that $G$ is a countable group such that $C^*(G)$ does not have the LP.  Set $D:=C^*(\mathbb{F}_\infty*G)$ and for each $n\in \mathbb N$, set $D_n=D$.  Finally, set $A:=\ast_J D_J$.  Then for any cofinal ultrafilter $\u$ on $\hat{\mathbb{N}}$, we have that $A$ does not have property max-$\u$.
\end{thm}

We end this section with a couple of further observations about property max-$\u$.  The following is the analog of Proposition \ref{minsubalgebra}:

\begin{prop}\label{maxsubalgebra}
If $B$ has property max-$\u$ and $A\subseteq B$ is r.w.i., then $A$ has property max-$\u$.
\end{prop}

\begin{proof}
This follows from the following diagram and Lemma \ref{rwiultra}.
\[
\begin{tikzcd}
B^\u\otimes_{\max} B^\u \arrow[r,hook] & (B\otimes_{\max}B)^\u\\
A^\u\otimes_{\max} A^\u \arrow[u,hook] \arrow[r]& (A\otimes_{\max}A)^\u\arrow[u,hook]
\end{tikzcd}
\]
\end{proof}

\begin{cor}
If $A$ has property max-$\u$ for all ultrafilters $\u$, then so does every \cstar-algebra elementarily equivalent to $A$.
\end{cor}

\begin{proof}
As was already pointed out, the inclusion of an algebra in its ultrapower is r.w.i.  Thus, by Proposition \ref{maxsubalgebra} and the Keisler-Shelah theorem, it suffices to show that if $A$ has property max-$\u$ for all $\u$, then so does $A^\mathcal{V}$ for all ultrafilters $\mathcal{V}$.  This follows from the following diagram:
\[
\begin{tikzcd}
(A^{\mathcal{V}})^\u\otimes_{\max}(A^{\mathcal{V}})^\u\arrow[r,"\cong"] \arrow[d] &A^{\mathcal{V}\times \u}\otimes_{\max} A^{\mathcal{V}\times \u}\arrow[r,hook] &(A\otimes_{\max}A)^{\mathcal{V}\times \u}\arrow["\cong",to=Z]\\
(A^{\mathcal{V}}\otimes_{\max}A^{\mathcal{V}})^\u \arrow[hook,to=Z]&|[alias=Z]|((A\otimes_{\max}A)^{\mathcal{V}})^\u
\end{tikzcd}
\]
Here, the inclusion in the top row uses the fact that $A$ has property max-$\mathcal{V}\times \u$ while the inclusion in the bottom row uses that $A$ has property max-$\mathcal{V}$.  Finally, the downward arrow uses the fact that there are natural commuting maps $(A^{\mathcal{V}})^\u\to (A^{\mathcal{V}}\otimes_{\max}A^{\mathcal{V}})^\u$ induced by the natural commuting maps $A^{\mathcal{V}}\to A^{\mathcal{V}}\otimes_{\max}A^{\mathcal{V}}$.
\end{proof}

\section{The local case}

\subsection{Preliminaries on nonlocal games}

Throughout this section, $k$ and $n$ denote natural numbers that are at least $2$ and $[k]$ denotes the set $\{1,\ldots,k\}$ (and likewise for $[n]$).

We recall the following definitions from quantum information theory.  First, recall that a \textbf{positive operator-valued measure (POVM)} on a Hilbert space $H$ is a sequence $(A_1,\ldots,A_n)$ of positive operators on $H$ such that $A_1+\cdots+A_n=I_H$; we refer to $n$ as the length of the POVM.\footnote{In the previous section, we were using $A$'s for \cstar-algebras; in the quantum complexity literature, $A$ is more commonly used in connection with a POVM.  Thus, in this section, we will switch notation and use $C$'s and $D$'s for \cstar-algebras.}

\begin{defn}

\

\begin{enumerate}
\item The set $\cqs(k,n)$ of \textbf{quantum correlations} consists of the correlations of the form $p(a,b|x,y)= \langle A^x_a \otimes B^y_b\xi,\xi \rangle$ for $x,y\in [k]$ and $a,b\in [n]$,
where H is a finite-dimensional Hilbert space, $\xi \in H \otimes H$ is a unit vector, and for every $x,y \in [k]$,
$(A^x_a: a\in [n])$ and $(B^y_b \ : \ b\in [n])$ are POVMs on H of length $n$.  We view $C_q(k,n)$ as a subset of $[0,1]^{k^2n^2}$.
\item We set $\cqa(n,m)$ to be the closure in $[0,1]^{k^2n^2}$ of $\cqs(k,n)$; the correclations in $\cqa(k,n)$ are called \textbf{quantum asymptotic correlations}.
\item The set $\cqc(k,n)$ of \textbf{quantum commuting correlations} consists of the correlations of the form $p(a,b|x,y)= \langle A^x_a B^y_b\xi,\xi \rangle$ for $x,y \in [k]$ and $a,b\in [n]$,
where H is a separable (possibly infinite-dimensional) Hilbert space, $\xi \in H$ is a unit vector, and for every $x,y\in [k]$,
$(A^x_a: a\in [n])$ and $(B^y_b \ : \ b\in [n])$ are POVMs on H for which $A^x_a B^y_b=B^y_bA^x_a$ for all $x,y\in [k]$ and $a,b\in [n]$.
\end{enumerate}
\end{defn}

Note that $C_{q}(k,n)\subseteq C_{qc}(k,n)$; since it is known that the latter set is closed, we in fact have that $C_{qa}(k,n)\subseteq C_{qc}(k,n)$.

\begin{defn}

\

\begin{enumerate}
    \item A \textbf{nonlocal game with $k$ questions and $n$ answers} is a pair $\frak G=(\pi,D)$, where $\pi$ is a probability distribution on $[k]\times [k]$ and $$D:[k]\times [k]\times [n]\times [n]\to \{0,1\}$$ is a function.
\item Given a nonlocal game $\frak G$ as in the previous item and $p\in [0,1]^{k^2n^2}$, we define the \textbf{value of $\frak G$ at $p$} to be $$\val(\frak G,p):=\sum_{x,y\in [k]}\pi(x,y)\sum_{a,b\in [n]}D(x,y,a,b)p(a,b|x,y).$$
\end{enumerate}
\end{defn}

\begin{defn}
Given a nonlocal game $\frak G$ with $k$ questions and $n$ answers, we define:
\begin{enumerate}
    \item the \textbf{entangled value of $\frak G$} to be 
    $$\val^*(\frak G):=\sup_{p\in \cqa(k,n)}\val(\frak G,p);$$
    \item the \textbf{quantum commuting value of $\frak G$} to be $$\val^{co}(\frak G):=\sup_{p\in \cqc(k,n)}\val(\frak G,p).$$
\end{enumerate}
\end{defn}

We shall refer to the following results as ``MIP*=RE'' and ``MIP$^{co}$=coRE'' although they are really consequences of these results.

\begin{thm}[\cite{MIP*}]\label{MIP*=RE}
There is no algorithm such that, upon input a nonlocal game $\frak G$ with the promise that $\val^*(\frak G)\leq 1/2$ or $\val^*(\frak G)=1$, decides which of the two cases holds. 
\end{thm}

\begin{thm}[\cite{MIPco}]\label{MIPco=coRE}
There is no algorithm such that, upon input a nonlocal game $\frak G$ with the promise that $\val^{co}(\frak G)\leq 1/2$ or $\val^*(\frak G)=1$, decides which of the two cases holds. 
\end{thm}

\begin{fact}

\

\begin{enumerate}
    \item There is an algorithm such that, upon input a nonlocal game $\frak G$, produces a computable sequence of lower bounds to $\val^*(\frak G)$ that converges to $\val^*(\frak G)$.
    \item There is an algorithm such that, upon input a nonlocal game $\frak G$, produces a computable sequence of upper bounds to $\val^{co}(\frak G)$ that converges to $\val^{co}(\frak G)$.
\end{enumerate}
\end{fact}

The first algorithm is a simple ``brute'' force search while the second algorithm uses the theory of semidefinite programming.  Theorems \ref{MIP*=RE} and \ref{MIPco=coRE} above show that in item (1) of the previous fact, there is no algorithm for producing a computable sequence of upper bounds converging to $\val^*(\frak G)$, and likewise in item (2), there is no algorithm for producing a computable sequence of lower bounds converging to $\val^{co}(\frak G)$.

We will also need correlation sets defined using \cstar-algebras as introduced in \cite{tsirelson}.  In what follows, a POVM in a \cstar-algebra $C$ is a sequence $A_1,\ldots,A_n\in C$ of positive elements of $C$ such that $A_1+\cdots+A_n=1_C$; once again, we refer to $k$ as the length of the POVM.

% The following, which we present as a definition, is equivalent to the definition give in terms of provers and quantum entanglement \cite{MIP*}.
% \begin{defn}
% A language $L$ (in the sense of complexity theory) belongs to the class $\mip^*$ if there is an effective mapping $z\mapsto \frak G_z$ from strings to nonlocal games such that:
% \begin{itemize}
%     \item if $z\in L$, then $\val^*(\frak G_z)\geq \frac{2}{3}$
%     \item if $z\notin L$, then $\val^*(\frak G_z)\leq \frac{1}{3}$.
% \end{itemize}
% The complexity class $\mip^{co}$ can be defined in an analogous fashion, using $\val^{co}$ instead of $\val^*$.
% \end{defn}

\begin{defn}
Let $C_{\min}(C,D,k,n)$ (respectively $C_{\max}(C,D,k,n))$ denote the \emph{closure} of the set of correlations of the form $\phi(A^x_a\otimes B^y_b)$, where $A^1,\ldots,A^k$ are POVMs of length $n$ from $C$, $B^1,\ldots,B^k$ are POVMs of length $n$ from $D$, and $\phi$ is a state on $C\otimes_{\min} D$ (respectively a state on $C\otimes_{\max} D$).
\end{defn}

The following facts were established in \cite{tsirelson}:
\begin{fact}\label{tsirelsonfacts}

\

\begin{enumerate}
\item $C_{\min}(C,D,k,n)\subseteq C_{\max}(C,D,k,n)$.
    \item $C_{\min}(C,D,k,n)\subseteq C_{qa}(k,n)$.  If neither $C$ nor $D$ are subhomogeneous, then $C_{\min}(C,D,k,n)= C_{qa}(k,n)$.
    \item $C_{\max}(k,n)\subseteq C_{qc}(k,n)$.
\end{enumerate}
\end{fact}

\begin{defn}
We say that $(C,D)$ is a \textbf{Tsirelson pair} if $C_{\min}(C,D,k,n)=C_{\max}(C,D,k,n)$ for all $(k,n)$.
\end{defn}

The following was also established in \cite{tsirelson}:

\begin{fact}
If either $C$ or $D$ has the QWEP property, then $(C,D)$ is a Tsirelson pair.
\end{fact}

If $(C,C)$ is a Tsirelson pair and $C$ is not subhomogeneous, then it follows that $C_{\min}(C,C,k,n)=C_{qa}(k,n)$.  We introduce the analogous notion here for maximal tensor products and qc-correlations:

\begin{defn}
We say that a \cstar-algebra $C$ is \textbf{qc-full} if $C_{\max}(C,C,k,n)=C_{qc}(k,n)$ for all $k,n\geq 2$.  We say that a group $G$ is qc-full if $C^*(G)$ is qc-full.
\end{defn}

We list some basic facts about qc-full \cstar-algebras:

\begin{lemma}

\

\begin{enumerate}
    \item Any surjectively universal \cstar-algebra is qc-full.  In particular, any surjectively universal group (such as $C^*(\mathbb{F}_\infty)$) is qc-full.
    \item If $C$ is qc-full and $C$ is r.w.i. in $D$, then $D$ is also qc-full.
    \item If $H$ is a subgroup of $G$ and $H$ is qc-full, then $G$ is also qc-full.  In particular, if $G$ is a group that contains $\mathbb{F}_\infty$ as a subgroup (such as $\mathbb{F}_n$ for any $n\geq 2$), then $G$ is qc-full.
\end{enumerate}
\end{lemma}

\begin{proof}
Item (1) is an immediate consequence of \cite[Lemma 2.4]{tsirelson}.  Item (2) follows from the fact that $C\otimes_{\max} C\subseteq D\otimes_{\max}D$ when $C$ is r.w.i. in $D$.  Item (3) is a specific instance of item (2) as $C^*(H)$ is r.w.i. in $C^*(G)$ when $H$ is a subgroup of $G$ (see \cite[Section 3]{OzQ}).
\end{proof}

We note also that if $(C,C)$ is a Tsirelson pair (such as when $C$ has the QWEP property) and $C$ is not subhomogeneous, then $C$ is not qc-full (for otherwise $C_{qa}(k,n)=C_{qc}(k,n)$ for all $k,n\geq 2$, contradicting MIP$^*$=RE).

\begin{question}\label{qcfullquestion}
Is the class of qc-full \cstar-algebras elementary?
\end{question}

In regards to the previous question, we make the following observations:

\begin{prop}

\

\begin{enumerate}
    \item The class of qc-full \cstar-algebras is closed under ultraproducts by regular ultrafilters.
    \item If $C$ is a elementary subalgebra of $D$ for which $D$ is separable and qc-full and $C$ has property max-$\u$ for some nonprincipal ultrafilter $\u$, then $C$ is also qc-full.
\end{enumerate}
\end{prop}

\begin{proof}
To prove (1), suppose that $(C_i)_{i\in I}$ is a family of qc-full \cstar-algebras and $\u$ is a regular ultrafilter on $I$.  Set $C:=\prod_\u C_i$.  Fix a correlation $p\in C_{qc}(k,n)$.  Since $\u$ is regular and each $C_i$ is qc-full, we may find POVMs $(A^x_a)_i,(B^y_b)_i\in C_i$ and states $\phi_i$ on $C_i\otimes_{\max}C_i$ such that $$\lim_\u \phi_i((A^x_a)_i\otimes (B^y_b)_i)=p(a,b|x,y).$$  Set $A^x_a:=((A^x_a)_i)_\u$ and $B^y_b:=((B^y_b)_i)_\u$, which are POVMs in $C$.  Moreover, let $\phi$ be the state on $C\otimes_{\max}C$ induced by the $*$-homomorphism $$C\otimes_{\max}C\to \prod_\u (C_i\otimes_{\max}C_i),$$ where the latter is equipped with the ultraproduct state $\lim_\u \phi_i$.  It follows that $p(a,b|x,y)=\phi(A^x_a\otimes B^y_b)$, whence $p\in C_{\max}(C,C,k,n)$.

To prove (2), once again fix $p\in C_{qc}(k,n)$.  Fix an elementary embedding $i:D\hookrightarrow C^\u$ that restricts to the diagonal embedding on $C$.  Consider the maps 
$$C\otimes_{\max}C\subseteq D\otimes_{\max}D\hookrightarrow (C^\u\otimes_{\max}C^\u)\subseteq (C\otimes_{\max}C)^\u.$$
Here, the first inclusion follows from the fact that $C$ is r.w.i. in $D$ (since $C$ is an elementary subalgebra of $D$) and similarly for the embedding $i\otimes i$ in the middle; the final inclusion follows from the assumption that $C$ has property max-$\u$.  Since $D$ is qc-full, given $\epsilon>0$, there are POVMS $A^x_a$ and $B^y_b$ in $D$ and a state $\phi$ on $D\otimes_{\max}D$ such that $|p(a,b|x,y)-\phi(A^x_a\otimes B^y_b)|<\epsilon$.  Via these embeddings, we may identify $A^x_a$ with $((A_x^a)_i)_u \in C^\u$ and similarly $B^y_b$ with $((B^y_b)_i)_\u\in C^\u$.  Moreover, we may extend $\phi$ to a state, which we also call $\phi$, on all of $(C\otimes_{\max}C)^\u$.  By \cite[Lemma 2.3]{tsirelson}, we may write the restriction of $\phi$ to the image of $D\otimes_{\max}D$ as $(\phi_i)_\u$, where each $\phi_i$ is a state on $C\otimes_{\max}C$.  It follows that, for $\u$-almost all $i\in I$, we have $|p(a,b|x,y)-\phi_i((A^x_a)_i\otimes (B^y_b)_i)|<\epsilon$ and we thus conclude that $C$ is qc-full.
\end{proof}

Note that the proof of part (2) of the previous proposition shows that the conclusion holds true when $D$ is not necessarily separable as long as $\u$ is a ``good enough'' ultrafilter so that $C^\u$ is sufficiently saturated so as to elementarily embed $D$.  The restriction to regular ultrafilters in (1) is not severe as it pertains to checking that the class of qc-full \cstar-algebras is axiomatizable.  In order to give a positive answer to Question \ref{qcfullquestion}, one would have to remove the assumption of property max-$\u$ in item (2).

Returning to the main thread, one can analogously define the values of a nonlocal game obtained by using strategies from the sets $C_{\min}(C,D,k,n)$ and $C_{\max}(C,D,k,n)$; we call these values $\val^{C\otimes D}(\mathfrak{G})$ and $\val^{C\otimes_{\max} D}(\mathfrak{G})$.  However, since $$\sum_{x,y\in [k]}\pi(x,y)\sum_{a,b\in[n]}D(x,y,a,b)A^x_a\otimes B^y_b$$ is a positive element in either $C\otimes D$ or $C\otimes_{\max}D$ and since the norm of a positive element in any \cstar-algebra is given by the supremum of all states applied to that element, we immediately see the following:

\begin{lemma}\label{gamevalues}
For any nonlocal game $\mathfrak{G}$ and any pair of \cstar-algebras $C$ and $D$, we have

$$\val^{C\otimes D}(\mathfrak{G})=\sup\{P^{\min}_{C,k^2n^2}(\pi(x,y)D(x,y,a,b)A^x_a,B^y_b) \ : \ A^x,B^y \text{ POVMs in }C\}$$
and 
$$\val^{C\otimes_{\max} D}(\mathfrak{G})=\sup\{P^{\max}_{C,k^2n^2}(\pi(x,y)D(x,y,a,b)A^x_a,B^y_b) \ : \ A^x,B^y \text{ POVMs in }C\}.$$
In particular, if $(C,C)$ is a Tsirelson pair and $C$ is not subhomogeneous (resp. if $C$ is qc-full), then $\val^{C\otimes C}(\mathfrak{G})=\val^*(\mathfrak{G})$ (resp. $\val^{C\otimes_{\max} C}(\mathfrak{G})=\val^{co}(\mathfrak{G})$).
\end{lemma}

We need one last fact about POVMs in \cstar-algebras.  Let $\rho_n(x_1,\ldots,x_n)$ denote the following existential formula in the language of \cstar-algebras:
$$\inf_{y_1}\cdots\inf_{y_n}\max\left(\max_{i=1,\ldots,n}\|x_i-y_iy_i^*\|,\left\|\sum_{i=1}^m x_i-1\right\|\right).$$
Note that the zeroset of $\rho_n$ in a \cstar-algebra $C$ is the set of POVMs of length $n$ in $C$.  We call a sequence $(A_1,\ldots,A_n)$ from a \cstar-algebra $C$ a \textbf{$\delta$-almost POVM of length $n$} if $\rho_n(A_1,\ldots,A_m)^C<\delta$.  It was shown in \cite{KEPJFA} that there is a computable function $\Delta_n:(0,1)\to (0,1)$ such that, for all $\epsilon>0$, all \cstar-algebras $C$, and all $\Delta_n(\epsilon)$-almost POVMS $(A_1,\ldots,A_n)$ from $C$, there is a POVM $B_1,\ldots,B_n$ from $C$ of length $n$ such that $\|A_i-B_i\|<\epsilon$ for all $i=1,\ldots,n$.

\subsection{The minimal tensor norm}

\begin{defn}
We say that $C$ has \textbf{weakly explicit minimal tensor norm} if there is an effective map such that, upon input $n$, returns a universal restricted\footnote{A restricted formula is one such that all connectives used in the formation of the formula are among the unary connectives $0$, $1$, $x\mapsto x/2$ and the binary connective $x\dminus y$.} formula $\varphi_n(\vec x,\vec y)$ such that, for all $\vec a,\vec b\in C^n$, we have $$\|P^{\min}_{C,n}(\vec a,\vec b)-\varphi_n(\vec a,\vec b)^C\|<1/8.$$  The notion of having \textbf{weakly explicit maximal tensor norm} is defined analogously.
\end{defn}

\begin{lemma}
If $D\subseteq C$ is existential and $C$ has weakly explicit minimal or maximal tensor norm, then so does $D$.
\end{lemma}

\begin{proof}
This follows from the definition and the fact that existential inclusions are r.w.i.
\end{proof}

\begin{thm}\label{minexplicit}
Suppose that $C$ is a locally universal \cstar-algebra.  Then $C$ does not have weakly explicit minimal tensor norm.
\end{thm}

\begin{proof}
Suppose, towards a contradiction, that $C$ has weakly explicit minimal tensor norm as witnessed by the formulae $\varphi_n$.  We claim then that we can decide if a given nonlocal game $\frak G=(\pi,D)$ (with $k$ questions and $n$ answers) is such that $\val^*(\frak G)\leq 1/2$ or $\val^*(\frak G)=1$, promised that one of them is the case, contradicting MIP*=RE.  Indeed, as stated above, one first runs the brute force algorithm for computing lower bounds for $\val^*(\frak G)$.  On the other hand, since $C$ is not subhomogeneous (as it is locally universal), by Lemma \ref{gamevalues} above, we have that 
$$\val^*(\frak G)=\sup\{P^{\min}_{C,k^2n^2}(\pi(x,y)D(x,y,a,b)A^x_a,B^y_b) \ : \ A^x,B^y \text{ POVMs in }C\}.$$
Letting $\mathcal{X}$ denote the set of $k^2n^2$-tuples that form tuples of POVMs of the appropriate length, we can rewrite the previous equality as
$$\val^*(\frak G)=\sup_{(A,B)}\left(P^{\min}_{C,k^2n^2}(\pi(x,y)D(x,y,a,b)A^x_a,B^y_b)\dminus d((A,B),\mathcal{X})\right).$$
Letting $\rho(X,Y)$ denote the obvious formula defining the set $\mathcal{X}$ (which uses the formulae $\rho_n$ introduced in the previous section), there is a computable, continuous, non-decreasing function $\alpha:[0,1)\to [0,1)$ such that $\alpha(0)=0$ and  
$$d((A,B),\mathcal{X})=\inf_{(A',B')}\left(d((A,B),(A',B'))+\alpha(\rho(A',B'))\right).$$  Plugging this expression for $d((A,B),\mathcal{X})$ into the above formula for $\val^*(\frak G)$ and approximating $P^{\min}_{C,k^2n^2}$ with $\varphi_{k^2n^2}$, we obtain a universal sentence $\sigma_{k^2n^2}$ such that $|\val^*(\frak G)-\sigma_{k^2n^2}^C|<1/8$.  Since $C$ is locally universal, we have, for any universal sentence $\sigma$, that $\sigma^C=\inf\{q\in \mathbb{Q}^{>0} \ : \ T\vdash \sigma\dminus q\}$, where $T$ is the theory of \cstar-algebras.  Now start running proofs from $T$ and record whenever $T\vdash \sigma\dotminus q$, for then we know that $\sigma^C_{k^2n^2}\leq q$; if we ever see this for some $q\leq 3/4$, we declare $\val^*(\frak G)\leq 1/2$.  

To see that the conjunction of these two algorithms works, note that if $\val^*(\frak G)=1$, then the brute force algorithm will eventually confirm the fact that $\val^*(\frak G)>1/2$, whence we can conclude that $\val^*(\frak G)=1$.  If, on the other hand, $\val^*(\frak G)\leq 1/2$, then $\sigma^C_{k^2n^2}\leq 5/8$ and  that fact will eventually be recorded.  To see that the second algorithm makes no mistakes, if we see that $\sigma^C_{k^2n^2}\leq 3/4$, then we know that $\val^*(\frak G)\leq 7/8$, and thus $\val^*(\frak G)\leq 1/2$.
\end{proof}

The same reasoning shows the following more general result:

\begin{thm}
Suppose that $T$ is an effectively axiomatizable, non-subhomogeneous theory.  If $C$ is a locally universal model of $T$, then $C$ does not have weakly explicit minimal tensor norm.
\end{thm}

\begin{question}\label{Cuntzquestion}
Does the Cuntz algebra $\mathcal{O}_2$ have weakly explicit (minimal) tensor norm?   
\end{question}

The \textbf{Kirchberg embedding problem (KEP)} asks whether or not $\mathcal{O}_2$ is locally universal; see \cite{KEPJFA} or \cite{KEP} for more on this problem.  If Question \ref{Cuntzquestion} has a positive answer, then by Theorem \ref{minexplicit}, we see that the KEP has a negative solution.

\subsection{The maximal tensor norm}

Theorem \ref{minexplicit} allows us to make some conclusions about algebras that do not have weakly explicit maximal tensor norm.

\begin{thm}
Suppose that $T$ is an effectively axiomatizable non-subhomogeneous theory.  If $C$ is a locally universal model of $T$ for which $(C,C)$ is a Tsirelson pair (such as when $C$ has QWEP), then $C$ does not have weakly explicit maximal tensor norm.  
\end{thm}

\begin{proof}
The only thing to note is that when $C$ is a Tsirelson pair, one has that 
$$\val^*(\frak G)=\sup\{P^{\max}_{C,k^2n^2}(\pi(x,y)D(x,y,a,b)A^x_a,B^y_b) \ : \ A^x,B^y \text{ POVMs in }C\}.$$
Then one argues as in the proof of Theorem \ref{minexplicit}.
\end{proof}

Recall that a \cstar-algebra is \textbf{existentially closed (e.c.)} if it is existential in all algebras containing it.  Since e.c. \cstar-algebras are locally universal and have the WEP \cite{KEP}, we can conclude:

\begin{cor}
    If $C$ is an e.c. \cstar-algebra, then $C$ does not have weakly explicit minimal tensor norm nor weakly explicit maximal tensor norm.
\end{cor}

\section{Computable presentations and maximal tensor products}

\subsection{Preliminaries on computable presentations of \cstar-algebras}

Let $C$ be a separable \cstar-algebra.  A \textbf{presentation} of $C$ is a pair $C^\dagger:=(C,(a_n)_{n\in \mathbb N})$, where $\{a_n \ : \ n\in \mathbb N\}$ is a subset of $C$ that generates $C$ (as a \cstar-algebra).  
Elements of the sequence $(a_n)_{n\in \mathbb N}$ are referred to as \textbf{special points} of the presentation while elements of the form $p(a_{i_1},\ldots,a_{i_k})$ for $p$ a $*$-polynomial with coefficients from $\mathbb Q(i)$ (a \textbf{rational polynomial}) are referred to as \textbf{generated points} of the presentation.  By fixing an effective bijection between the set of rational polynomials and $\mathbb N$, we can fix an effective enumeration of the rational polynomials and thus from any presentation of a \cstar-algebra we obtain an effective enumeration of the generated points of the presentation.  If $x$ is the $n^{\text{th}}$ generated point of $C^\dagger$, then we call $n$ an \textbf{code} for $x$.

We say that $C^\dagger$ is a \textbf{computable presentation} of $C$ if there is an algorithm such that, upon input a generated point $p$ of $C^\dagger$ and $k\in \mathbb N$, returns a rational number $q$ such that $|\|p\|-q|<2^{-k}$.  A weaker notion is that of a c.e. presentation, which means that there is an algorithm which, upon input a generated point $p$ of $C^\dagger$, enumerates a sequence of upper bounds which converges to $\|p\|$.  If $C^\dagger$ is a computable or c.e. presentation, then by a \textbf{code} for $C^\dagger$ we mean a natural number which codes the finite sequence of strings that describes the algorithm.  

An element $x\in C$ is a \textbf{computable point} of the presentation $C^\dagger$ if there is an algorithm which, upon input $k\in \mathbb N$, returns a generated point $q$ of $C^\dagger$ such that $\|x-q\|<2^{-k}$.  Once again, one can speak of the code of a computable point of a presentation.

If $C^\dagger$ and $D^\#$ are presentations of \cstar-algebras $C$ and $D$ respectively, a function $\varphi : C^\dagger \to D^\#$ is \textbf{computable} if $\varphi$ is a function from $C$ to $D$ for which there is is an algorithim such that, upon input a generated point $p$ of $C^\dagger$ and $k \in \mathbb{N}$, returns a generated point $q$ of $D^\#$ such that $\|\varphi(p) - q\| < 2^{-k}$; in other words, $\varphi$ is a computable map if the $\varphi$-images of generated points of $C^\dagger$ are computable points of $D^\#$, uniformly in the code for the generated point, meaning that the code for the computable point $\varphi(p)$ can be computed from the code for $p$.  Once again, one may speak of the code of a computable map as the code of such an algorithm.  An isomorphism $\varphi:C\to D$ between \cstar-algebras is a \textbf{computable isomorphism} from $C^\dagger$ to $D^\#$ if it is a computable map with computable inverse.  (The computability of the inverse is automatic if $D^\#$ is computable.) 

There is a nice connection between presentations (in the group theoretic sense) of a countable group $G$ and presentations of its universal group \cstar-algebra $C^*(G)$; a thorough discussion of this connection is described in \cite[Section 7.2.8]{EMST}.  We mention here briefly the facts relevant to the discussion below.  

Given a countable group $G$ with distinguished generating set $X$, there is a corresponding presentation of $C^*(G)$, where the special points are those given by the elements of $X$ (or rather the canonical unitaries associated with these group elements).  Given groups $G$ and $H$ with distinguished generating sets $X$ and $Y$, it is straightforward to see that the canonical maps $C^*(G)\to C^*(G\times H)$ and $C^*(H)\to C^*(G\times H)$ are computable, when all of the algebras involved are given their presentations corresponding to the associated generating sets (and where $G\times H$ is given the generating set $(X\times \{1_H\})\cup (\{1_G\}\times Y)$).  Moreover, the isomorphism $C^*(G\times H)\cong C^*(G)\otimes_{\max} C^*(H)$ induces a presentation of $C^*(G)\otimes_{\max}C^*(H)$; by the previous sentence, it follows that the canonical maps $C^*(G)\to C^*(G)\otimes_{\max} C^*(H)$ and $C^*(H)\to C^*(G)\otimes_{\max}C^*(H)$ are computable, again when all of the algebras are equipped with the appropriate presentations.  For a more general discussion of presentations of tensor products of \cstar-algebras, see \cite[Subsection 2.3]{ssa}.  

While for many finitely generated groups $(G,X)$, the above presentation of $C^*(G)$ is merely c.e., there are instances when the presentation is in fact computable.  For example, consider the finitely generated free group $\mathbb{F}_n$ with its standard generating set.  In \cite[Theorem 1.5]{canyou}, it is shown that the corresponding presentation of $C^*(\mathbb{F}_n)$ is computable.  Moreover, the algorithm for determining the norm is uniform in $n$.  As a result, the presentation of $C^*(\mathbb{F}_\infty)$ corresponding to the standard generating set for $\mathbb{F}_\infty$ is also computable.  Indeed, given any generated point of the standard presentation of $\mathbb{F}_\infty$, one may first search for an $n$ for which the generated point is also a generated point of the standard presentation of $\mathbb{F}_n$ and then approximate the norm there.  Since the norm of the element in $C^*(\mathbb{F}_n)$ and in $C^*(\mathbb{F}_\infty)$ are the same, this yields the desired algorithm. 

\subsection{Applications to maximal tensor products}

\begin{thm}\label{maxexplicit}
Suppose that $C$ is qc-full and has a computable presentation $C^\dagger$.  Then $C$ does not have a weakly explicit maximal tensor norm.
\end{thm}

\begin{proof}
Suppose, towards a contradiction, that $C$ has a weakly explicit maximal tensor norm as witnessed by the formulae $\varphi_n$.  For simplicity, we set $\varphi(A,B)$ to be the formula $\varphi_{k^2n^2}(\pi(x,y)D(x,y,a,b)A^x_a, B^y_b)$.  We claim that we can decide if a given nonlocal game $\frak G=(\pi,D)$ is such that $\val^{co}(\frak G)\leq 1/2$ or $\val^{co}(\frak G)=1$, promised that one of them is the case, contradicting MIP$^{co}$=coRE.  Indeed, first, as discussed above, run the semidefinite programming algorithm which gives upper bounds for $\val^{co}(\frak G)$; if that algorithm ever gives an upper bound below $1$, one can conclude that $\val^{co}(\frak G)\leq 1/2$.  On the other hand, start searching for tuples $A$ and $B$ consisting of generated points of $C^\dagger$ that form $\gamma$-approximate POVMs in $C$ of the appropriate length and computing lower bounds for $\varphi(A,B)$, which is possible since $\varphi_{k^2n^2}$ is universal and the presentation $C^\dagger$ is computable; here $\gamma=\Delta_n(\Delta_\varphi(1/8))$ and $\Delta_\varphi$ is the modulus of uniform continuity for the formula $\varphi$.  If this search ever sees a value above $7/8$, then the algorithm declares that $\val^{co}(\frak G)=1$.

To see that this algorithm works, suppose first that $\val^{co}(\frak G)=1$.  Then there is a correlation $p(x,y|a,b)\in \cqc(k,n)$ realizing this fact.  By assumption, such a correlation is of the form $\phi(A^x_a\otimes B^y_b)$ for some tuples of POVMs $A$ and $B$ in $C$ and state $\phi$ on $C\otimes_{\max}C$.  By Lemma \ref{gamevalues}, we have that $P_{C,k^2n^2}^{\max}(\pi(x,y)D(x,y,a,b)A^x_a, B^y_b)=1$ and thus $\varphi(A,B)>7/8$.  By approximating $A$ and $B$ by sequences $A'$ and $B'$ of suitably close generated points that form $\gamma$-approximate POVMs, we will see that $\varphi(A',B')>7/8$.  To see that this algorithm makes no mistakes, if $\varphi(A',B')>7/8$ on some $\gamma$-approximate POVMs $A',B'$ in $C$, then by the choice of $\gamma$, one can perturb them to actual POVMs $A$ and $B$ in $C$ for which $\varphi(A,B)>3/4$, and then $P_{C,k^2n^2}^{\max}(\pi(x,y)D(x,y,a,b)A^x_a, B^y_b)>1/2$, implying that $\val^{co}(\frak G)>1/2$ and thus $\val^{co}(\frak G)=1$. 
\end{proof}

As mentioned above, for any $n\geq 2$ (including $n=\infty$), $\mathbb{F}_n$ is qc-full.  Moreover, as mentioned in the previous subsection, $C^*(\mathbb{F}_n)$ has a computable presentation.  As a result, we obtain the following:

\begin{cor}
For any $n\geq 2$ (including $n=\infty$), $C^*(\mathbb{F}_n)$ does not have weakly explicit maximal tensor norm.
\end{cor}

As mentioned above, any surjectively universal \cstar-algebra $C$ is qc-full.  We thus conclude:

\begin{cor}
If $C$ is a separable surjectively universal \cstar-algebra, then either $C$ does not have a computable presentation or else $C$ does not have weakly explicit maximal tensor norm.
\end{cor}

The proof of Theorem \ref{maxexplicit} also shows the following:

\begin{thm}\label{comptensor}
Suppose that $C$ is qc-full.  Then there do not exist computable presentations $C^\dagger$ of $C$ and $(C\otimes_{\max}C)^\dagger$ of $C\otimes_{\max}C$ for which the mappings $a\mapsto a\otimes 1,a\mapsto 1\otimes a:C^\dagger\to (C\otimes_{\max}C)^\dagger$ are computable.  
\end{thm}

\begin{proof}
As in the proof of Theorem \ref{maxexplicit}, since $C^\dagger$ is a computable presentation, one can search for sequences of generated points of $C^\dagger$ that form approximate POVMs in $C$.  Moreover, the other assumptions of the theorem allow one to approximate $P^{\max}_{C,k^2,n^2}$ on tuples of generated points of $C^\dagger$.  The rest of the proof remains the same.
\end{proof}

We note that \cite[Subsection 2.3]{ssa} gives some very general conditions which ensure that the maps $a\mapsto a\otimes 1$ and $a\mapsto 1\otimes a$ are computable. 

Recalling the general discussion about presentations of group \cstar-algebras in the previous subsection together with the fact that $\mathbb{F}_n$ is qc-full, a particular consequence of Theorem \ref{comptensor} is the following:

\begin{cor}\label{FNTquestion}
For any $n\geq 2$ (including $n=\infty$), the standard presentation of $C^*(\mathbb{F}_n\times \mathbb{F}_n)$ is not computable.  
\end{cor}

The previous corollary (for $n=2$) answers a question posed by Fritz, Netzer, and Thom in \cite{canyou}.

By \cite[Theorem 3.3]{fox}, the standard presentation of $C^*(\mathbb{F}_n\times \mathbb{F}_n)$ is c.e.  It would be interesting to determine if $C^*(\mathbb{F}_n\times \mathbb{F}_n)$ is a solution to the following seemingly open question:

\begin{question}
Is there a separable \cstar-algebra that has a c.e. presentation but no computable presentation?
\end{question}

It is a well-known open question (due to Ozawa \cite[Section 3]{OzQ}) whether or not $C^*(\mathbb{F}_n\times \mathbb{F}_n)$ has the LP.  Fixing a surjective *-homomorphism $$C^*(\mathbb{F}_\infty)\to C^*(\mathbb{F}_n\times \mathbb{F}_n),$$ this is equivalent to asking if the identity map on $C^*(\mathbb{F}_n\times \mathbb{F}_n)$ admits a ucp lift $C^*(\mathbb{F}_n\times \mathbb{F}_n)\to C^*(\mathbb{F}_\infty)$.  Corollary \ref{FNTquestion} allows us to conclude the following result, which perhaps explains why the LP for $C^*(\mathbb{F}_n\times \mathbb{F}_n)$ is so difficult:

\begin{cor}
Given a surjective *-homomorphism $C^*(\mathbb{F}_\infty)\to C^*(\mathbb{F}_n\times \mathbb{F}_n)$, there is no \emph{computable} ucp lift $C^*(\mathbb{F}_n\times \mathbb{F}_n)\to C^*(\mathbb{F}_\infty)$ (where both sides are equipped with their standard presentations).
\end{cor}

For $n\geq 2$, and $\epsilon>0$, define the ``softenings'' $C^*(\mathbb{F}_n\times \mathbb{F}_n)_\epsilon$ to be the universal \cstar-algebra generated by unitaries $u_1,\ldots,u_n,v_1,\ldots,v_n$ satisfying $\|[u_i,v_j]\|\leq \epsilon$ for all $i,j=1,\ldots,n$.  In \cite[Proposition 6.10]{ES}, it was shown that $C^*(\mathbb{F}_n\times \mathbb{F}_n)_\epsilon$ is RFD, whence, by \cite[Theorem 3.5]{fox}, the standard presentations of these algebras are computable, uniformly in $\epsilon$.  Since these algebras form an inductive system whose inductive limit is $C^*(\mathbb{F}_n\times \mathbb{F}_n)$, we have the following consequence of Corollary \ref{FNTquestion}:

\begin{cor}
Fix $n\geq 2$.  For each *-polynomial $p(x_1,\ldots,x_n,y_1,\ldots,y_n)$ and $m\geq 1$, set $\alpha_{m,p}:=\|p(\vec u,\vec v)\|_{C^*(\mathbb{F}_n\times \mathbb{F}_n)_{1/m}}$ and $\alpha_{p}:=\|p(\vec u,\vec v)\|_{C^*(\mathbb{F}_n\times \mathbb{F}_n)}$.  Then there is no algorithm such that, upon input $p$ and $q\in \mathbb{Q}^{>0}$, determines an $m$ such that $\|\alpha_{m,p}-\alpha_p\|<q$.  
\end{cor}

\end{document}